\documentclass[12pt]{amsart}
\usepackage[english]{babel}
\usepackage{amsthm,amssymb,amsmath,amsfonts,mathrsfs, mathtools}
\usepackage{amsmath}
\usepackage{geometry}
\usepackage[matrix,arrow]{xy}
\usepackage{graphicx}
\theoremstyle{plain}
 \newtheorem{theorem}{Theorem}
\theoremstyle{plain}
 \newtheorem{proposition}{Proposition}
 \theoremstyle{definition}
 \newtheorem{definition}{Definition}

\begin{document}

\title{Harmonic Spheres in the Hilbert--Schmidt Grassmannian}
\author{Iuliya BELOSHAPKA, Armen SERGEEV}

\address{Moscow State University, Steklov Mathematical Institute}
\date{}
\thanks{While preparing this paper the second author was partly supported
by the RFBR grant 13-01-00622, the Leading Scientific Schools
program (grant NSh-2928.2012.1), and Scientific Program of Presidium
of RAS ''Nonlinear dynamics''.}

\dedicatory{To Sergei Petrovich Novikov on the occasion of his 75th
birthdate}

\date{}


\begin{abstract}
We give the twistor description of harmonic maps of the Riemann
sphere into the Hilbert--Schmidt Grassmannian. The study of such
maps is motivated by the harmonic spheres conjecture formulated in
the beginning of this paper.
\end{abstract}

\maketitle

\bigskip

%-------------------------------------------------------------------------

\section*{Introduction}

%-------------------------------------------------------------------------

A well known theorem of Atiyah \cite{Ati} establishes a 1--1
correspondence between the moduli space of $G$-instantons on
4-dimensional Euclidean space $\mathbb R^4$ and the space of based
holomorphic maps of the Riemann sphere $\mathbb P^1$ into the loop
space $\Omega G$ of a compact Lie group $G$. The harmonic spheres
conjecture, obtained from this formulation by ''realification'',
asserts that it should exist a natural 1--1 correspondence between
the moduli space of Yang--Mills $G$-fields on $\mathbb R^4$ and the
space of based harmonic maps of the Riemann sphere $\mathbb P^1$
into the loop space $\Omega G$.

The proof of the Atiyah theorem depends on the Donaldson theorem
\cite{Don}, establishing a 1--1 correspondence between the moduli
space of $G$-instantons on $\mathbb R^4$ and the set of equivalence
classes of holomorphic $G^\mathbb C$-bundles over $\mathbb
P^1\times\mathbb P^1$ which are trivial on the union $\mathbb
P^1_\infty\cup\,\mathbb P^1_\infty$ of projective lines at
''infinity''. The Donaldson theorem may be considered as a
2-dimensional reduction of the well known Atiyah--Ward theorem,
relating instantons on $\mathbb R^4$ with holomorphic bundles on the
3-dimensional projective space $\mathbb P^3$ which are trivial on
the fibres of the twistor bundle $\mathbb P^3\setminus\mathbb
P^1_\infty\to\mathbb R^4$. The proof of Donaldson theorem is based
on the monad method, used for the construction of holomorphic vector
bundles on projective spaces, and is ''purely complex''. So to prove
the harmonic spheres conjecture one should look for a ''real''
analogue of Donaldson theorem for harmonic bundles.

Such an analogue was proposed in \cite{bogo100} where it was also
outlined an idea of the proof of the harmonic spheres conjecture.
One of its essential ingredients is the twistor description of
harmonic spheres in the Hilbert--Schmidt Grassmannian which is the
main subject of this paper.

\bigskip

%-------------------------------------------------------------------------

\section{Motivation: harmonic spheres conjecture}
\label{sec1}

%-------------------------------------------------------------------------

The motivation to study harmonic maps of the Riemann sphere into the
Hilbert--Schmidt Grassmannian comes from the harmonic spheres
conjecture which is formulated below.

We start by recalling some necessary facts about instantons and
Yang--Mills fields on the one hand, and harmonic maps on the other
hand.

%-------------------------------------------------------------------------

\subsection{Instantons and Yang--Mills fields}
\label{ssec11}

%-------------------------------------------------------------------------

\medskip

Let $G$ be a compact Lie group and $A$ is a $G$-connection on
$\mathbb R^4$ given by the 1-form
$$
A=\sum_{\mu=1}^4 A_\mu(x)dx_\mu
$$
with smooth coefficients $A_\mu(x)$, taking values in the Lie
algebra $\mathfrak g$ of $G$. Denote by $F_A$ the curvature of $A$
given by the 2-form
$$
F_A=\sum_{\mu,\nu=1}^4 F_{\mu\nu}(x)dx_\mu\wedge dx_\nu
$$
with coefficients
$$
F_{\mu\nu}=\partial_\mu A_\nu-\partial_\nu A_\mu+[A_\mu,A_\nu]
$$
where $\partial_\mu:=\partial/\partial x_\mu$, $\mu=1,2,3,4$, and
$[\cdot,\cdot]$ denotes the commutator in Lie algebra $\mathfrak g$.

Define the \textit{Yang--Mills action} functional by the formula
$$
S(A)=\frac12\int_{\mathbb R^4}\text{tr}(F_A\wedge *F_A)
$$
where $*$ is the Hodge operator on $\mathbb R^4$, and the trace
$\text{tr}$ is computed with the help of a fixed invariant inner
product on the Lie algebra $\mathfrak g$.

The functional $S(A)$ is invariant under the \textit{gauge
transformations} given by
$$
A\longmapsto A_g:=g^{-1}dg+g^{-1}Ag
$$
where $g:\mathbb R^4\to G$ is a smooth map, and $G$ acts on its Lie
algebra $\mathfrak g$ by the adjoint representation.

The extremals of the functional $S(A)$ with finite action
$S(A)<\infty$ are called the \textit{Yang--Mills fields}.

Yang--Mills fields have an integer-valued topological invariant,
called the \textit{topological charge}, computed by the formula
$$
k(A)=\frac1{8\pi^2}\int_{\mathbb R^4}\text{tr}(F_A\wedge F_A).
$$
If we write $F_A$ in the form
$$
F_A=F_++F_-
$$
where $F_{\pm}=\frac12(* F_A\pm F_A)$ then the formulae for the
action and charge may be rewritten as
\begin{align*}
S(A) &=\frac12\int_{\mathbb R^4}\left(\Vert F_+\Vert^2+\Vert
F_-\Vert^2\right)d^4x,\\
k(A) &=\frac1{8\pi^2}\int_{\mathbb R^4}\left(-\Vert F_+\Vert^2+\Vert
F_-\Vert^2\right)d^4x
\end{align*}
where the norm $\Vert\cdot\Vert^2$ is computed with the help of
invariant inner product on the Lie algebra $\mathfrak g$.

By comparing the last two formulae we arrive at the inequality
$$
S(A)\geq 4\pi^2|k(A)|
$$
where the equality is attained for $k>0$ on solutions of the
equation
\begin{equation}
\label{eq1}
* F_A=-F_A,
\end{equation}
and for $k<0$ on solutions of the equation
\begin{equation}
\label{eq2} * F_A=F_A.
\end{equation}

\begin{definition}
Solutions of equation \eqref{eq1} with $S(A)<\infty$ are called
\textit{instantons}, and solutions of equation \eqref{eq2} with
$S(A)<\infty$ are called \textit{anti-instantons}.
\end{definition}

Instantons and anti-instantons realize local minima of the action
$S(A)$, however, there exist also non-minimal critical points of
this functional.

\bigskip

%-------------------------------------------------------------------------

\subsection{Harmonic spheres} \label{ssec12}

%-------------------------------------------------------------------------

\medskip

Let $\varphi:\mathbb P^1\to N$ be a smooth map from the Riemann
sphere $\mathbb P^1$ into an oriented Riemannian manifold $N$. We
call such a map \textit{harmonic} if it is extremal with respect to
the energy functional given by the following Dirichlet integral
$$
E(\varphi)=\frac12\int_{\mathbb C}|d\varphi|_N^2\frac{|dz\wedge
d\bar z|}{(1+|z|^2)^2}
$$
where the modulus of differential $d\varphi$ is computed with
respect to the metric $|\cdot|_N$ of manifold $N$.

If the manifold $N$ is K\"ahler, i.e. $N$ has a complex structure,
compatible with the Riemannian metric, then holomorphic and
anti-holomorphic maps $\varphi:\mathbb P^1\to N$ realize local
minima of the energy functional $E(\varphi)$. However, for
$\text{dim}_{\mathbb C}N>1$ there exist usually also non-minimal
harmonic maps.

Comparing harmonic maps with Yang-Mills fields, we notice an evident
analogy between:
$$
\left\{\parbox{4,3cm}{(anti)holomorphic
maps}\right\}\longleftrightarrow
\left\{\parbox{2,9cm}{(anti)instantons}\right\}
$$
and
$$
\left\{\parbox{2,8cm}{harmonic maps}\right\}\longleftrightarrow
\left\{\parbox{3,1cm}{Yang--Mills fields}\right\}\ .
$$
We are going to give a mathematical justification of this
observation.

%-------------------------------------------------------------------------

\subsection{Twistor interpretation of instantons}
\label{ssec13}

%-------------------------------------------------------------------------

\medskip

The twistor interpretation of instantons is based on the following
\textit{twistor bundle}
\begin{equation}
\label{fibre2} \pi: \mathbb P^3\setminus\mathbb
P^1_\infty\longrightarrow\mathbb R^4
\end{equation}
over the Euclidean space $\mathbb R^4$ where $\mathbb P^3$ is the
3-dimensional complex projective space (cf. \cite{Ati-book}). The
fibre of this bundle at $x\in\mathbb R^4$ can be identified with the
space of complex structures on the tangent space $T_x\mathbb
R^4\cong\mathbb R^4$, compatible with metric and orientation.

In terms of the twistor bundle $\pi: \mathbb P^3\setminus\mathbb
P^1\to\mathbb R^4$ the \textit{moduli space of $G$-instantons}, i.e.
the quotient of the space of all $G$-instantons on $\mathbb R^4$
modulo gauge transformations, admits the following interpretation
given by the \textit{Atiyah--Ward theorem}:
$$
\left\{\parbox{3,5cm}{moduli space of $G$-instatons on $\mathbb
R^4$}\right\}\longleftrightarrow \left\{\parbox{8cm}{(based)
equivalence classes of holomorphic $G^{\mathbb C}$-bundles over
$\mathbb P^3$, holomorphically trivial on $\pi$-fibers}\right\}\ .
$$
This result has the following 2-dimensional reduction to the space
$\mathbb P^1\times\mathbb P^1$ given by the Donaldson theorem:
$$
\left\{\parbox{3,7cm}{moduli space of $G$-instantons on $\mathbb
R^4$}\right\}\longleftrightarrow \left\{\parbox{8,7cm}{(based)
equivalence classes of holomorphic $G^{\mathbb C}$-bundles over
$\mathbb P^1\times\mathbb P^1$, holomorphically trivial on the union
$\mathbb P^1_\infty\cup\mathbb P^1_\infty$}\right\}\ .
$$

%-------------------------------------------------------------------------

\subsection{Twistor interpretation of harmonic spheres}
\label{ssec14}

%-------------------------------------------------------------------------

\medskip

Using the interpretation of the twistor bundle over $\mathbb R^4$ as
the bundle of compatible complex structure, one can extend its
definition to any even-dimensional oriented Riemannian manifold $N$.
Namely, the twistor bundle $\pi:Z\to N$ over $N$ is defined as the
bundle of complex structures on $N$, compatible with metric and
orientation. The fibre of this bundle can be identified with the
homogeneous space $\text{SO}(2n)/\text{U}(n)$ where $2n$ is the
dimension of $N$. The twistor space $Z$ can be provided, as it was
shown in \cite{AHS}, with a natural almost complex structure,
denoted by $\mathcal J^1$.

However, for the description of harmonic spheres in $N$ we have to
employ another almost complex structure which is defined in the
following way. The Levi-Civita connection on $N$ generates a
connection on the twistor bundle $\pi:Z\to N$. The new almost
complex structure on $Z$, denoted by $\mathcal J^2$, is equal to
$\mathcal J^1$ in the directions, horizontal with respect to the
introduced connection, and to $-\mathcal J^1$ in vertical
directions. This structure was introduced in \cite{Ee-Sal} and is
always non integrable. Harmonic spheres in $N$ admit the following
description in terms of this almost complex structure.

\begin{theorem}[Eells--Salamon \cite{Ee-Sal}]
\label{ee-sal} Projections $\varphi=\pi\circ\psi$ of the maps
$\psi:\mathbb P^1\to Z$, holomorphic with respect to the almost
complex structure $\mathcal J^2$, are harmonic.
\end{theorem}

This theorem allows to construct harmonic spheres in the manifold
$N$ from almost holomorphic spheres in its twistor space $Z$. And
this is an idea of the twistor approach to harmonic maps --- to try
to reduce the original ''real'' problem of construction of harmonic
spheres in the Riemannian manifold $N$ to the ''complex'' problem of
construction of holomorphic spheres in the almost complex manifold
$Z$.

We show in Section~\ref{ssec22} how this idea can be applied to the
description of harmonic spheres in the complex Grassmann manifold
$G_r(\mathbb C^d)$.

%-------------------------------------------------------------------------

\subsection{Atiyah theorem}
\label{ssec15}

%-------------------------------------------------------------------------

\medskip

We switch now to the case when the target Riemannian manifold $N$ is
infinite-dimensional, more precisely, the loop space of a compact
Lie group.

Recall first the definition of this space. Let $G$ be a compact Lie
group. Consider the quotient
$$
\Omega G:=LG/G
$$
of the loop group $LG=C^\infty(S^1,G)$ modulo the subgroup of
constant maps $S^1\to g_0\in G$ identified with the group $G$. The
space $\Omega G$ is a Frechet manifold which has an $LG$-invariant
complex structure (cf. \cite{Ser-book}). This structure is induced
from the representation of $\Omega G$ as the quotient of the complex
loop group $LG^\mathbb C$:
$$
\Omega G=LG^\mathbb C/L_+G^\mathbb C
$$
where $G^\mathbb C$ is the complexification of the group $G$ and the
subgroup $L_+G^\mathbb C$ consists of the loops $\gamma\in
LG^\mathbb C$ which can be smoothly extended to holomorphic maps of
the unit disc $\Delta$ into $G^\mathbb C$. Apart from the complex
structure $\Omega G$ has also an $LG$-invariant symplectic structure
(cf. \cite{Ser-book}) which is compatible with the introduced
complex structure in the sense that they generate together a natural
Riemannian metric on $\Omega G$.

Recall that by Donalson theorem the moduli space of $G$-instantons
can be identified with:
$$
\left\{\parbox{3,7cm}{moduli space of $G$-instantons on $\mathbb
R^4$}\right\}\longleftrightarrow \left\{\parbox{8,7cm}{(based)
equivalence classes of holomorphic $G^{\mathbb C}$-bundles over
$\mathbb P^1\times\mathbb P^1$, holomorphically trivial on the union
$\mathbb P^1_\infty\cup\mathbb P^1_\infty$}\right\}\ .
$$
Atiyah theorem asserts that the right hand side of this
correspondence may be identified with the space of holomorphic
spheres in $\Omega G$. In more detail, there is a 1--1
correspondence between:
$$
\left\{\parbox{6cm}{(based) equivalence classes of holomorphic
$G^{\mathbb C}$-bundles over $\mathbb P^1\times\mathbb P^1$,
holomorphically trivial on the union $\mathbb P^1_\infty\cup\mathbb
P^1_\infty$}\right\}\longleftrightarrow \left\{\parbox{5,5cm}{based
holomorphic spheres $f:\mathbb P^1\to\Omega G$, sending $\infty$
into the origin of $\Omega G$}\right\}\ .
$$

\bigskip

%-------------------------------------------------------------------------

\subsection{Harmonic spheres conjecture}
\label{ssec16}

%-------------------------------------------------------------------------

\medskip

The Donaldson and Atiyah theorems imply that there is a 1--1
correspondence between:
$$
\left\{\parbox{3,5cm}{moduli space of $G$-instantons on $\mathbb
R^4$}\right\}\longleftrightarrow \left\{\parbox{5cm}{based
holomorphic spheres $f:\mathbb P^1\to\Omega G$}\right\}\ .
$$
So we have a correspondence between local minima of two functionals,
introduced before, namely the Yang--Mills action, defined on gauge
$G$-fields on $\mathbb R^4$, and the energy functional, defined on
smooth spheres in the loop space $\Omega G$. Recall that the local
minima of Yang--Mills action are given by instantons and
anti-instantons on $\mathbb R^4$ while the local minima of energy
are given by holomorphic and anti-holomorphic spheres in $\Omega G$.
Replacing local minima by critical points of these functionals, we
arrive at the \textit{harmonic spheres conjecture} asserting that it
should exist a 1--1 correspondence between:
$$
\left\{\parbox{5cm}{moduli space of Yang--Mills $G$-fields on
$\mathbb R^4$}\right\}\longleftrightarrow
\left\{\parbox{4,8cm}{based harmonic spheres $f:\mathbb P^1\to\Omega
G$}\right\}\ .
$$
The described transition from the local minima to critical points
may be considered as a kind of the ''realification'' procedure.
Indeed, if we replace smooth spheres in the right hand side of the
diagram by smooth functions $f:\mathbb C\to\mathbb C$ then the above
procedure will reduce to the replacement of holomorphic and
anti-holomorphic functions by arbitrary harmonic functions (which
can be represented as sums of holomorphic and anti-holomorphic
functions). In the case of smooth spheres in $\Omega G$ this
transition from holomorphic and anti-holomorphic spheres to harmonic
ones becomes non-trivial due to the non-linear character of the
Euler--Lagrange equations for the energy.

Unfortunately, a direct generalization of Atiyah--Donaldson theorem
to the harmonic case is not possible since the proof of Donaldson
theorem, using the monad method, is purely holomorphic and does not
extend directly to the harmonic case. However, in \cite{bogo100} it
was proposed another way of proving the formulated conjecture.
Namely, one can try to reduce the proof of harmonic spheres
conjecture to the holomorphic case by ''pulling-up'' the both sides
of the correspondence in this conjecture to the associated twistor
spaces.

One of the key points in this approach is an extension of the
twistor construction of harmonic maps, presented in
Sec.~\ref{ssec14}, to the case of infinite-dimensional target
spaces. More precisely, we are interested in the
infinite-dimensional generalization of the construction of harmonic
maps into the complex Grassmann manifold $G_r(\mathbb C^d)$,
presented in Section~\ref{ssec22}.

\bigskip

%-------------------------------------------------------------------------

\section{Harmonic maps into complex Grassmanninans}
\label{sec2}

%-------------------------------------------------------------------------

In this Section we study harmonic maps from the Riemann surfaces
into the complex Grassmann manifolds $G_r(\mathbb C^d)$.

%-------------------------------------------------------------------------

\subsection{Harmonic maps into Riemannian manifolds}
\label{ssec21}

%-------------------------------------------------------------------------

\medskip
Recall first some general properties of harmonic maps of Riemannian
manifolds.

Let $\varphi:(M,g)\to(N,h)$ be a smooth map of a Riemannian manifold
$M$ with Riemannian metric $g$ into a Riemannian manifold $N$ with
Riemannian metric $h$. We define the \textit{energy} of the map
$\varphi$ as the Dirichlet integral
\begin{equation}
\label{en} E(\varphi)=\frac{1}{2}\int_M\vert
d\varphi(p)\vert^2\text{vol}_g\,.
\end{equation}
The norm of the differential may be computed in local coordinates as
follows. Denote by $(x^i)$ local coordinates at $p\in M$ and by
$(u^{\alpha})$ local coordinates at $q=\varphi(p)\in N$. Then
$$
\vert
d\varphi(p)\vert^2=\sum_{i,j}\sum_{\alpha,\beta}g^{ij}\,\frac{\partial
\varphi^{\alpha}}{\partial x^i}\,\frac{\partial\varphi^{\beta}}
{\partial x^j}\,h_{\alpha\beta}
$$
where $\varphi^{\alpha}=\varphi^{\alpha}(x)$ are the components of
$\varphi$, $(g_{ij})$ and $(h_{\alpha\beta})$ are metric tensors of
$M$ and $N$ respectively, and $(g^{ij})$ is the inverse matrix of
$(g_{ij})$, $\text{vol}_g$ is the volume element of metric $g$.

A smooth map $\varphi:M\to N$ is called \textit{harmonic} if it is
extremal for the functional $E(\varphi)$ with respect to all smooth
variations of $\varphi$ with compact supports.

The Euler--Lagrange equation for the energy functional $E(\varphi)$,
called otherwise the \textit{harmonic map equation}, has the
following form in local coordinates $(x^i)$ on $M$ and
$(u^{\alpha})$ on $N$:
\begin{equation}
\label{HM} \Delta_M\varphi^{\gamma} + \sum_{i,j}g^{ij}
\sum_{\alpha,\beta}\,^N\Gamma_{\alpha\beta}^{\gamma}(\varphi)
\frac{\partial\varphi^{\alpha}}{\partial
x_i}\frac{\partial\varphi^{\beta}} {\partial x_j} = 0
\end{equation}
where $\Delta_M$ is the standard Laplace--Beltrami operator on $M$,
given by
$$
\Delta_M\varphi^{\gamma}=
\sum_{i,j}g^{ij}\left\{\frac{\partial^2\varphi^{\gamma}}{\partial
x_i\partial x_j} - \sum_k\,^M\Gamma_{ij}^k
\frac{\partial\varphi^{\gamma}}{\partial x_k}\right\}\;.
$$
Here, ${}^M\Gamma_{ij}^k$ denotes the Christoffel symbol of the
Levi-Civita connection ${}^M\nabla$ of $M$; respectively,
${}^N\Gamma_{\alpha\beta}^{\gamma}$ is the Christoffel symbol of the
Levi-Civita connection ${}^N\nabla$ of $N$.

In particular case $N=\mathbb R^n$ the equation \eqref{HM} becomes
linear and reduces to the Laplace--Beltrami equation
$$
\Delta_M\varphi^{\gamma}=0\ ,\quad\gamma=1,\dots,n
$$
on the components of $\varphi$.

Suppose that the Riemannian manifold $(M,g)$ is provided with a
complex (or almost complex) structure ${}^MJ$, compatible with
Riemannian metric $g$, and, likewise, the target manifold $(N,h)$
has a complex (or almost complex) structure ${}^NJ$, compatible with
Riemannian metric $h$.

Recall that a smooth map $\varphi:M\to N$ is called
\textit{(pseudo)holomorphic} if the tangent map $\varphi_*: TM\to
TN$ commutes with (almost) complex structures on $M$ and $N$, i.e.\
$$
\varphi_*\circ{}^MJ = {}^NJ\circ\varphi_*\;,
$$
and it is called \textit{anti-(pseudo)holomorphic} if $\varphi_*$
anti-commutes with the (almost) complex structures on $M$ and $N$.

The complexified tangent bundle $T^{\mathbb C}M=TM\otimes\mathbb C$
can be decomposed into the direct sum
$$
T^{\mathbb C}M=T^{1,0}M\oplus T^{0,1}M
$$
of subbundles with the fibres, given by the $(\pm i)$-eigenspaces of
the almost complex structure operator ${}^MJ$. If we extend the
tangent map $\varphi_{\ast}$ complex linearly to the complexified
tangent bundles, then we obtain a map $\varphi_{\ast}\colon
T^{\mathbb C}M\to T^{\mathbb C}N$, which, in accordance with the
above decomposition, splits into the sum of four operators
\begin{align}
\label{eq11}
\partial^{\prime}\varphi\colon T^{1,0}M\to T^{1,0}N &\quad,\quad
\partial^{\prime\prime}\varphi\colon T^{0,1}M\to T^{1,0}N\ ,\\
\partial^{\prime}\bar\varphi=\overline{\partial^{\prime\prime}
\varphi}\colon T^{1,0}M\to T^{0,1}N &\quad, \quad
\partial^{\prime\prime}\bar\varphi=\overline{\partial^{\prime}
\varphi}\colon T^{0,1}M\to T^{0,1}N\ .
\end{align}

The introduced operators may be considered as sections of the bundle
$(T^\ast M)^{\mathbb C}\otimes\varphi^{-1}(T^{\mathbb C}N)$. In
these notations the map $\varphi$ is pseudoholomorphic (resp.
anti-pseudo\-holomorphic) if and only if
$\partial^{\prime\prime}\varphi=0$
(resp.$\partial^{\prime}\varphi=0$).

It may be proved (cf. \cite{Ee-Sam}) that for the (almost) K\"ahler
manifolds (pseudo)holo\-morphic and anti-(pseudo)holomorphic maps
$\varphi:M\to N$ always realize local minima of the energy
functional $E(\varphi)$ but, in general, there exist another
critical points of $E(\varphi)$, i.e.\ non-minimal harmonic maps.

We restrict now to the case, when $M$ is a compact Riemann surface.
Denote by $\nabla$ the connection on the bundle
$\varphi^{-1}(T^{\mathbb C}N)$ over $M$, induced by the Levi--Civita
connection ${}^N\nabla$ on the Riemannian manifold $N$. If $z$ is a
local complex coordinate on $M$, we set
$$
\delta\varphi=\varphi_*(\partial/\partial z),\quad
\bar\delta\varphi=\varphi_*(\partial/\partial\bar z)
$$
where $\delta\varphi$ and $\bar\delta\varphi$ are considered as
sections of the bundle $\varphi^{-1}(T^{\mathbb C}N)$. (More
generally, we denote by $\delta=\nabla_{\partial/\partial z}$,
$\bar\delta= \nabla_{\partial/\partial\bar z}$ the components of the
connection $\nabla$.) The differential $d\varphi$ is represented in
the form
$$
d\varphi=dz\otimes\delta\varphi+d\bar z\otimes\bar\delta\varphi,
$$
and harmonic map equation \eqref{HM} may be rewritten in the form
\begin{equation}
\label{harm-eq1} \bar\delta\delta\varphi=
\left(\nabla_{\partial/\partial\bar z}\varphi_*\right)
\left(\frac{\partial}{\partial z}\right)=0\ .
\end{equation}
Equivalent formulation
$$
\delta\bar\delta\varphi= \left(\nabla_{\partial/\partial z}\varphi_*
\right) \left(\frac{\partial}{\partial\bar z}\right)=0\ .
$$

If $N$ is a K\"ahler manifold, then, according to \eqref{eq11},
$$
\delta\varphi=\partial'\varphi+\overline{\partial''\varphi}\ ,\quad
\bar\delta\varphi=\partial''\varphi+\overline{\partial'\varphi}\ ,
$$
and harmonic map equation for $\varphi$ takes on the form %%
\begin{equation}
\label{harm-eq2} \bar\delta\partial'\varphi=0\ ,
\end{equation}
or, equivalently, as
$$
\delta\partial''\varphi=0\ .
$$

According to the Koszul--Malgrange theorem (cf. \cite{K-M}), any
complex vector bundle $E$ over a Riemann surface $M$ with a
connection $\nabla$ has a unique complex structure $J$, such that
$E\to M$ is a holomorphic vector bundle with respect to $J$, for
which the $\bar\partial_J$-operator coincides with the
$(0,1)$-component $\nabla^{0,1}$ of the connection $\nabla$. This
complex structure $J$ is called the \textit{KM-structure}.

In its terms, the harmonicity condition \eqref{harm-eq1} means that
$\delta\varphi$ is a holomorphic section of the bundle
$\varphi^{-1}(T^{\mathbb C}N)$ with respect to the KM-structure on
$\varphi^{-1}(T^{\mathbb C}N)$, induced by the connection
${}^N\nabla$. In the same way, the condition \eqref{harm-eq2} means
that $\partial'\varphi$ is a holomorphic section of the bundle
$\varphi^{-1}(T^{1,0}N)$.

\bigskip

%-------------------------------------------------------------------------

\subsection{Complex Grassmann manifolds and flag bundles}
\label{ssec22}

We are going to apply the twistor approach to the construction of
harmonic spheres, presented in Sec.~\ref{ssec14}, to the case when
the target manifold $N$ coincides with the complex Grassmannian
$G_r(\mathbb C^d)$. In this case it is natural to choose for the
twistor bundle over $G_r(\mathbb C^d)$ the bundle of complex
structures on $G_r(\mathbb C^d)$, invariant under the action of
unitary group $\text{U}(d)$. Such bundles coincide with flag bundles
over $G_r(\mathbb C^d)$ which we define next.

We start by recalling the definition of flag manifolds in $\mathbb
C^d$. For that we fix a decomposition of $d$ into the sum of natural
numbers $d=r_1+\dots+r_n$ and denote $\mathbf r:=(r_1,\dots,r_n)$.

The \textit{flag manifold} $F_{\,\mathbf r}(\mathbb C^d)$ in
$\mathbb C^d$ of \textit{type} $\mathbf r=(r_1,\ldots,r_n)$ with
$d=r_1+\ldots+r_n$ consists of flags $\mathcal W=(W_1,\ldots, W_n)$,
i.e. nested sequences of complex subspaces
$$
W_1\subset\ldots\subset W_n=\mathbb C^d,
$$
such that the dimension of the subspace $E_1:=W_1$ is equal to $r_1$
and dimensions of the subspaces $E_i:=W_i\ominus W_{i-1}$ are equal
to $r_i$ for $1<i\leq n$.

In particular, for $\mathbf r=(r,d-r)$ the flag manifold
$$
\mathcal F_{(r,d-r)}(\mathbb C^d) = \{ \mathcal E=(E,E^\perp): \dim
E=r\} =G_r(\mathbb C^d)
$$
coincides with the Grassmann manifold of $r$-dimensional subspaces
in $\mathbb C^d$.

We have the following homogeneous representation of the flag
manifold
$$
\mathcal F_{\,\mathbf r}(\mathbb C^d)=
\left.\text{U}(d)\right/\text{U}(r_1)\times\dots\times\text{U}(r_n)\;.
$$
There is also another, complex homogeneous representation for this
manifold
$$
\mathcal F_{\,\mathbf r}(\mathbb C^d)= \left.\text{GL}(d,\mathbb C)
\right/\mathcal P_{\mathbf r},
$$
where $\mathcal P_{\,\mathbf r}$ is the parabolic subgroup of
blockwise upper-triangular matrices of the form
$$
\begin{pmatrix}
\;\begin{tabular}{c|l} \quad
\raisebox{-0.5mm}{$*$}\vphantom{\bigg|}\quad\,\
\\\hline\end{tabular}\,\,\scriptstyle r_1 &\ast\quad\quad
&\quad\,*&\dots   &\ast
\\[-1.5mm]
\scriptstyle r_1\;\; &&&&\\[-3.2mm]
0\quad            & \hskip-4mm\begin{tabular}{|c|}\hline \;\quad
\raisebox{-0.2mm}{$*$}\vphantom{\Bigg|}\quad\;\,\
\\\hline\end{tabular}\,\,\scriptstyle r_2 &\quad\;\ast
&\quad\dots\quad   &\ast
\\[-1.5mm]
&\scriptstyle r_2\quad\;\; &&&\\
\vdots\quad       &            &\quad\;\ddots  &        &\vdots       \\
&&&&\ \ \scriptstyle r_n \\[-0.5mm]
0\quad            &0\quad\;\;           &\quad\,0       &\dots   &
\scriptstyle r_n\,\begin{tabular}{|c}\hline \quad
\raisebox{-0.2mm}{$*$}\vphantom{\bigg|}\quad\,\ \\\end{tabular}
\end{pmatrix}
$$%

\smallskip
\par\noindent
with blocks of dimensions $r_i\times r_i$ in the boxes.

These representations imply that $\mathcal F_{\,\mathbf r}(\mathbb
C^d)$ has a natural complex structure, which we denote by $J^1$.
Moreover, $\mathcal F_{\,\mathbf r}(\mathbb C^d)$, provided with
this complex structure, is a compact K\"ahler manifold.

In particular case $\mathbf r=(r,d-r)$ we obtain the well known
homogeneous representations for the Grassmann manifold
$$
G_r(\mathbb C^d)=\text{U}(d)/\text{U}(r)\times\text{U}(d-r)=
\text{GL}(d,\mathbb C)/P_{(r,d-r)}\ .
$$

We construct now a series of homogeneous flag bundles over the
Grassmann manifold $G_r(\mathbb C^c)$. Let $\mathcal F=\mathcal
F_{\,\mathbf r}(\mathbb C^N)$ be a flag manifold of type $\mathbf
r=(r_1,\dots,r_n)$ in $\mathbb C^d$ with the homogeneous
representation
$$
\mathcal F=\mathcal F_{\,\mathbf r}(\mathbb C^N)=
\text{U}(d)/\text{U}(r_1)\times\dots\times\text{U}(r_n)\;.
$$
On the Lie algebra level this representation corresponds to the
decomposition of the complexified Lie algebra $\mathfrak u^\mathbb
C(d)$ into the direct orthogonal sum %%
\begin{multline}
\mathfrak u^\mathbb C(d)\cong\mathfrak g\mathfrak l(d,\mathbb
C)\cong\overline{\mathbb C^d}\otimes\mathbb C^d\cong \left(\bar
E_1\oplus\dots\oplus\bar E_n\right)\otimes
\left(E_1\oplus\dots\oplus E_n\right)\cong\\
\cong\left[\mathfrak u^\mathbb C(r_1)\oplus\dots\oplus\mathfrak
u^\mathbb C(r_n)\right]\oplus \left[\bigoplus_{i<j}\left(\bar
E_iE_j\oplus\bar E_jE_i\right)\right]\ .
\end{multline}
(In the latter formula we have omitted the sign of the tensor
product in the expression $\bar E_iE_j$ and its conjugate in order
to make the formulas more visible. The same rule will be applied in
the sequel.)

The above decomposition of the Lie algebra $\mathfrak u^\mathbb
C(d)$ implies that the complexified tangent space $T_o^\mathbb
C\mathcal F$ at the origin $o\in\mathcal F$ coincides with
$$
T_o^\mathbb C\mathcal F=\bigoplus\limits_{i<j}\,D_{ij}^\mathbb C:=
\bigoplus\limits_{i<j}\left(\bar E_iE_j\oplus\bar E_jE_i\right)\ .
$$
Every component $D_{ij}$ may be provided with two different complex
structures: for one of them its $(1,0)$-subspace coincides with
$\bar E_iE_j$, for another with $\bar E_jE_i$. By the
Borel--Hirzebruch theorem \cite{Bo-Hir}, any $U(d)$-invariant almost
complex structure $J$ on $\mathcal F$ is determined by the choice of
one of these two complex structures on every $D_{ij}$. The almost
complex structure $J^1$, for which
$$
T_o^{1,0}\mathcal F=\bigoplus\limits_{i<j}\,\bar E_iE_j\;,
$$
is called \textit{canonical}.

Fix an ordered subset $\sigma\subset\{1,\ldots,n\}$ and set
$r:=\sum_{i\in\sigma}r_i$. We can associate with any such subset
$\sigma$ the homogeneous bundle %%
\begin{equation}
\label{twi-bun} \pi_{\sigma}\colon\ \mathcal F_{\,\mathbf r}(\mathbb
C^N)= \frac{U(d)}{U(r_1)\times\ldots\times U(r_n)} \longrightarrow
\frac{U(d)}{U(r)\times U(d-r)}=G_r(\mathbb C^d)
\end{equation}
by mapping $(E_1,\ldots,E_n)\longmapsto
E=\bigoplus_{i\in\sigma}E_i$.

The complexified tangent bundle $T^{\mathbb C}\mathcal F_{\,\mathbf
r}(\mathbb C^N)$ is decomposed into the direct sum of vertical and
horizontal subbundles. Namely, the vertical subspace at the origin
coincides with $\bigoplus\limits_{i,j}\,D^\mathbb C_{ij}$, where
$i<j$ and either $i,j\in\sigma$, or $i,j\notin\sigma$. Respectively,
the horizontal subspace at the origin is equal to
$\bigoplus\limits_{i,j}D^\mathbb C_{ij}$, where $i<j$ and either
$i\in\sigma,j\notin\sigma$, or $i\notin\sigma,j\in\sigma$.

We introduce, along with the canonical complex structure $\mathcal
J^1$, a new $\text{U}(d)$-invariant almost complex structure
$\mathcal J^2_\sigma$ on $\mathcal F_{\,\mathbf r}(\mathbb C^N)$, by
setting it equal to $\mathcal J^2_\sigma=\mathcal J^1$ on horizontal
tangent vectors and $\mathcal J^2_\sigma=-\mathcal J^1$ on vertical
tangent vectors.

With this almost complex structure on $\mathcal F_{\,\mathbf
r}(\mathbb C^N)$ we have the following analogue of the
Eells--Salamon theorem from Sec.~\ref{ssec14}.

\begin{theorem}[Burstall-Salamon]
\label{th111} The flag bundle
$$
\pi_\sigma: (F_{\,\mathbf r}(\mathbb C^d),\mathcal
J^2_{\sigma})\longrightarrow G_r(\mathbb C^d),
$$
provided with the almost complex structure $\mathcal J^2_{\sigma}$,
is a twistor bundle, i.e. the projection
$\varphi=\pi_\sigma\circ\psi$ of any holomorphic map $\psi:M\to
F_{\,\mathbf r}(\mathbb C^d)$ from the Riemann surface $M$ to
$G_r(\mathbb C^d)$ is a harmonic map $\varphi:M \to G_r(\mathbb
C^d)$. In the case when $M=\mathbb P^1$ is the Riemann sphere, the
converse of this assertion is also true: any harmonic sphere
$\varphi:\mathbb P^1\to G_r(\mathbb C^d)$ in $G_r(\mathbb C^d)$ may
be obtained as the projection of a holomorphic sphere in some flag
bundle $\pi_\sigma: F_{\,\mathbf r}(\mathbb C^d)\to G_r(\mathbb
C^d)$.
\end{theorem}

In Section~\ref{ssec32} we shall generalize this theorem to harmonic
maps with values in Hilbert--Schmidt Grassmannians.

%-------------------------------------------------------------------------

\section{Harmonic maps into Hilbert--Schmidt Grassmanninans}
\label{sec3}

%-------------------------------------------------------------------------

In this Section we prove our main result -- an infinite-dimensional
generalization of the Burstall--Salamon theorem from
Section~\ref{ssec22}.

%-------------------------------------------------------------------------

\subsection{Hilbert--Schmidt Grassmannians}
\label{ssec31}

%-------------------------------------------------------------------------

\medskip

Consider a complex Hilbert space $H$, provided with a
\textit{polarization}. It means that $H$ is decomposed
\begin{equation}
\label{pol}
H=H_+\oplus H_-
\end{equation}
into the direct orthogonal sum of closed infinite-dimensional
subspaces $H_\pm$. For example, in the case of the Hilbert space
$H=L^2_0(S^1,\mathbb C)$ of square integrable functions on $S^1$
with zero average one can take for such subspaces
$$
H_\pm=\{\gamma\in H: \gamma=\sum_{\pm k>0}\gamma_ke^{ik\theta}\}.
$$

Any bounded linear operator $A\in L(H)$ with respect to the
polarization \eqref{pol} can be written in the block form
$$
A= \begin{pmatrix} a & b\\ c & d
\end{pmatrix} =
\begin{pmatrix}
a: H_+\to H_+\ , & b: H_-\to H_+ \\ c: H_+\to H_-\ , & d: H_-\to H_-
\end{pmatrix}\;.
$$

Denote by $\text{GL}(H)$ the group of linear bounded operators in
$H$, having a bounded inverse, and introduce the
\textit{Hilbert--Schmidt group} $\text{GL}_{\text{HS}}(H)$,
consisting of operators $A\in\text{GL}(H)$, for which the
"off-diagonal" terms $b$ and $c$ are Hilbert--Schmidt operators. In
other words, the group $\text{GL}_{\text{HS}}(H)$ consists of
operators $A\in\text{GL}(H)$, for which the "off-diagonal" terms $b$
and $c$ are "small" with respect to the "diagonal" terms $a$ and
$d$. We denote also by $\text{U}_{\text{HS}}(H)$ the intersection of
$\text{GL}_{\text{HS}}(H)$ with the group $\text{U}(H)$ of unitary
operators in $H$.

As in finite-dimensional situation, there is a Grassmann manifold
$\text{Gr}_{\text{HS}}(H)$, called the Hilbert--Schmidt
Grassmannian, related to the group $\text{GL}_{\text{HS}}(H)$.

The \textit{Hilbert--Schmidt Grassmannian}
$\text{Gr}_{\text{HS}}(H)$ consists of all closed subspaces
$W\subset H$ such that the orthogonal projection $\text{pr}_+:W\to
H_+$ is a Fredholm operator, while the orthogonal projection
$\text{pr}_-:W\to H_-$ is a Hilbert--Schmidt operator. Equivalently:
a subspace $W\in\text{Gr}_{\text{HS}}(H)$ iff it coincides with the
image of a linear operator $w: H_+\to H$ such that
$w_+:=\text{pr}_+\circ w$ is a Fredholm operator, and
$w_-:=\text{pr}_-\circ w$ is a Hilbert--Schmidt operator.

In other words, $\text{Gr}_{\text{HS}}(H)$ consists of the subspaces
$W\subset H$, which differ "little" from the subspace $H_+$ in the
sense that projection $\text{pr}_+:W\to H_+$ is "close" to an
isomorphism and projection $\text{pr}_-:W\to H_-$ is "small".

We have the following homogeneous space representation of
$\text{Gr}_{\text{HS}}(H)$:
$$
\text{Gr}_{\text{HS}}(H)=\left.\text{U}_{\text{HS}}(H)\right/
\text{U}(H_+)\times\text{U}(H_-).
$$

Since $\text{U}_{\text{HS}}(H)$ acts transitively on the
Grassmannian $\text{Gr}_{\text{HS}}(H)$, we can construct an
$\text{U}_{\text{HS}}(H)$-invariant K\"ahler metric on
$\text{Gr}_{\text{HS}}(H)$ from an inner product on the tangent
space $T_{H_+}\text{Gr}_{\text{HS}}(H)$ at the origin
$H_+\in\text{Gr}_{\text{HS}}(H)$, invariant under the action of the
isotropy subgroup $\text{U}(H_+)\times\text{U}(H_-)$. The tangent
space $T_{H_+}\text{Gr}_{\text{HS}}(H)$ coincides with the space of
Hilbert--Schmidt operators $\text{HS}(H_+,H_-)$ and invariant inner
product on it is given by the formula
$$
(A,B)\longmapsto\text{Re}\left\{\text{tr}(AB^\dagger)\right\},\quad
A,B\in\text{HS}(H_+,H_-)\;.
$$
Note that the imaginary part of the complex inner product
$\text{tr}(AB^\dagger)$ determines a non-degenerate invariant
$2$-form on $T_{H_+}\text{Gr}_{\text{HS}}(H)$, which extends to an
$\text{U}_{\text{HS}}(H)$-invariant symplectic structure on the
manifold $\text{Gr}_{\text{HS}}(H)$. Hence, we have a K\"ahler
structure on $\text{Gr}_{\text{HS}}(H)$, which makes it a K\"ahler
Hilbert manifold.

There is an evident obstacle to the immediate extension of the
results, obtained in Section~\ref{ssec22} for finite-dimensional
Grassmannians, to $\text{Gr}_{\text{HS}}(H)$. Namely, all subspaces
$W\in\text{Gr}_{\text{HS}}(H)$ are infinite-dimensional, i.e. in
some sense have the same infinite "dimension" and the problem is how
to compare them. Fortunately, there is a replacement of the notion
of dimension in Hilbert--Schmidt Grassmannian, called the ''virtual
dimension'', which can be used for the comparison of different
subspaces.

In more detail, $\text{Gr}_{\text{HS}}(H)$ has a countable number of
connected components, numerated by the index of the Fredholm
operator $w_+$ for a subspace $W\in\text{Gr}_{\text{HS}}(H)$. We say
that $W$ has the \textit{virtual dimension} $r$ if the index of
$w_+$ is equal to $r$. Denote by $\text{G}_r(H)$ the component of
$\text{Gr}_{\text{HS}}(H)$, consisting of subspaces $W$ of virtual
dimension $r$.

Then we have the following decomposition of
$\text{Gr}_{\text{HS}}(H)$ into the disjoint union of its connected
components $\text{G}_r(H)$:
\begin{equation}
\label{union} \text{Gr}_{\text{HS}}(H)=\bigcup_r\,\text{G}_r(H)\;.
\end{equation}
\smallskip
\par\noindent Due to this decomposition, the study of harmonic maps
of Riemann surfaces into $\text{Gr}_{\text{HS}}(H)$ is reduced to
the study of harmonic maps into Grassmannians $\text{G}_r(H)$ of
virtual dimension $r$, which may be carried on along the same lines,
as in the case of Grassmann manifold $G_r(\mathbb C^d)$.

%-------------------------------------------------------------------------

\subsection{Harmonic maps into Hilbert--Schmidt Grassmannians}
\label{ssec32}

%-------------------------------------------------------------------------

Let $M$ be a Riemann surface. Denote by $M\times H$ the trivial
bundle $M\times H\to M$ where $H$ is a complex Hilbert space
provided with the polarization \eqref{pol}. We shall consider the
subbundles $E\subset M\times H$, having the fibres
$E_p\in\text{Gr}_{\text{HS}}(H)$ for $p\in M$. Any bundle $E$ of
this type defines a map
$\varphi_E:M\longrightarrow\text{Gr}_{\text{HS}}(H)$ by setting:
$\varphi_E(p):=\text{the fibre $E_p$ at $p\in M$}$. Conversely, any
map $\varphi: M\to\text{Gr}_{\text{HS}}(H)$ defines a subbundle
$E\subset M\times\text{Gr}_{\text{HS}}(H)$.

Consider a smooth map of a Riemann surface $M$ into Grassmannian
$\text{Gr}_{\text{HS}}(H)$. Denote by $\pi$ and $\pi^\perp$ the
orthogonal projections of $M\times H$ onto the subbundle $E$ and its
orthogonal complement $E^\perp$ respectively. The bundle $E$ is
provided with the complex KM-structure which is determined in a
local chart on $M$ with local coordinate $z$ by the
$\bar\partial$-operator
$$
\partial''_E=\pi\circ\frac{\partial}{\partial\bar z}\circ\pi.
$$
The inverse image $\varphi_E^{-1}(T^\mathbb
C\text{Gr}_{\text{HS}}(H))$ of the complexified tangent bundle of
$\text{Gr}_{\text{HS}}(H)$ under the map $\varphi_E$ admits the
decomposition
$$
\varphi_E^{-1}(T^\mathbb C\text{Gr}_{\text{HS}}(H))\cong
\overline{E}E^\perp\oplus\overline{E^\perp}E\;.
$$
In terms of this decomposition the differential of $\varphi_E$ has
local components
$$
A'_E:=\pi^\perp\circ\frac{\partial}{\partial z}\circ\pi\ ,\quad
A''_E:=\pi^\perp\circ\frac{\partial}{\partial\bar z}\circ\pi\;.
$$
(In the sequel we sometimes omit the sign $\circ$ to simplify the
formulas.) In particular, the bundle $E\subset M\times H$ is
holomorphic $\iff$ $A''_E=0$, and in this case the KM-structure on
$E$ coincides with the complex structure, induced from $M\times H$.
Then %%
\begin{multline*} 0=\pi^\perp\left[\frac{\partial}{\partial
z}(\pi+\pi^\perp)\frac{\partial}{\partial\bar z} -
\frac{\partial}{\partial\bar z}(\pi+\pi^\perp)\frac{\partial}
{\partial z}\right]\pi= \\
=A'_E\partial''_E+\partial'_{E^\perp}A''_E-A''_E\partial'_E-
\partial''_{E^\perp}A'_E=A'_E\partial''_E-\partial''_{E^\perp}A'_E\;.
\end{multline*}
In other words, $A'_E\in\text{Hom}(E,E^\perp)$ is holomorphic with
respect to the KM-structures on $E$ and $E^\perp$.

In general, we call a bundle $E\subset M\times H$ \textit{harmonic}
if
$$
A'_E\circ\partial''_E=\partial''_{E^\perp}\circ A'_E\;.
$$
The harmonicity of $E$ is equivalent to the harmonicity of the map
$\varphi_E:  M\to G_r(H)$ (cf. \cite{Bur-Wood}). Note also that $E$
is harmonic $\iff$ its orthogonal complement $E^\perp$ is harmonic.

Now we generalize this situation to the maps into
infinite-dimensional flag manifolds.

%-------------------------------------------------------------------------

\subsection{Holomorphic maps into Hilbert--Schmidt flag manifolds}
\label{ssec33}

%-------------------------------------------------------------------------

Introduce first the Hilbert--Schmidt flag manifolds. For that fix an
$n$-tuple $\mathbf r=(r_1,\ldots,r_n)$ of integers. The
\textit{Hilbert--Schmidt flag manifold} $F_{\mathbf r}(H)$ consists
of the flags of the form
$$
\mathcal E\equiv (E_1,\ldots E_n)
$$
where $E_k\equiv W_{\text{in}}$, $E_l\equiv W_{\text{out}}$ are
closed infinite-dimensional subspaces in $H$ and
$$
E_1,\ldots,E_{k-1},E_{k+1},\ldots,E_{l-1},E_{l+1},\ldots,E_n
$$
are finite-dimensional subspaces having the following properties:
\begin{enumerate}
\item the projection $\text{pr}_+:W_{\text{in}}\to H_+$ is a Fredholm
operator of index $r_k$ while the projection
$\text{pr}_-:W_{\text{in}}\to H_-$ is a Hilbert--Schmidt operator;
\item the projection $\text{pr}_-:W_{\text{out}}\to H_-$ is a Fredholm
operator of index $r_l$ while the projection
$\text{pr}_+:W_{\text{out}}\to H_+$ is a Hilbert--Schmidt operator;
\item $E_i$ with $i=1,\ldots,k-1,k+1,\ldots,l-1,l+1,\ldots,n$ are
$r_i$-dimensional vector subspaces in $H$;
\item all subspaces $E_i$ with $i=1,\ldots,n$ are
pairwise orthogonal and their direct sum is equal to $H$:
$E_1\oplus\ldots\oplus E_n=H$.
\end{enumerate}

To simplify the notation, we say that $\mathcal E=(E_1,\ldots E_n)$
is a virtual flag of virtual dimension $\mathbf r=(r_1,\ldots,r_n)$
having in mind that $r_k$ (resp. $r_l$) are integers, equal to the
virtual dimension of $E_k= W_{\text{in}}$ (resp.
$E_l=W_{\text{out}}$), while other $r_i$'s are positive integers,
equal to the dimensions of $E_i$'s for $i\neq k,l$.

The tangent space to $F_{\mathbf r}(H)$ at the origin, as in the
finite-dimensional case, is the direct sum of four different terms
\begin{multline*}
\label{decomp}
 T^{\mathbb{C}}(F_{\mathbf r}(H)) \cong
\bigoplus_{1\leq i,j\neq k,l,i<j\leq n}\left[\overline{E}_{i}E_{j}
\oplus\overline{E}_{j}E_{i}\right] \oplus\bigoplus_{1\leq i\leq
n,i\neq k}\left[\overline{W}_{\text{in}}E_{i}
\oplus\overline{E}_{i}W_{\text{in}}\right]\\ \oplus \bigoplus_{1\leq
i\leq n, i\neq l}\left[\overline{W}_{\text{out}}E_{i}
\oplus\overline{E}_{i}W_{\text{out}}\right]
\oplus\left[\overline{W}_{\text{in}}
W_{\text{out}}\oplus\overline{W}_{\text{out}}W_{\text{in}}\right].
\end{multline*}

The tangent space to $\text{Gr}_{\text{HS}}(H))$ looks the same if
we set $E_{i}=0$ for all $E_i$'s except for $i=k,l$.

Generalizing the situation, studied in Sec.~\ref{ssec32}, let us
consider an arbitrary collection $\mathcal E=(E_1,\ldots,E_n)$ of
mutually orthogonal subbundles $E_i$ in $M\times H$ of virtual
dimension $\mathbf r=(r_1,\ldots,r_n)$, generating the decomposition
of $M\times H$ into the direct orthogonal sum
$$
M\times H = \bigoplus\limits_{i=1}^n E_i.
$$
We call such a collection of subbundles $\mathcal
E=(E_1,\ldots,E_n)$ the \textit{moving flag} on $M$. It determines,
in the same way as before, a map $\psi_{\mathcal E}: M\to F_{\mathbf
r}(H)\equiv\mathcal F$ by assigning to a point $p\in M$ the flag,
defined by the subspaces $(E_{1,p},\ldots,E_{n,p})$. Conversely, any
smooth map $\psi\colon M\to\mathcal F$ determines a moving flag
$\mathcal E=(E_1,\ldots,E_n)$, where $E_i=\psi^{-1}T_i$ is the
pull-back of a natural tautological bundle $T_i\to F_{\mathbf
r}(H)$: the fibre of $T_i$ at $\mathcal E\in\mathcal F$ coincides,
by definition, with the subspace $E_i$ for $1\leq i\leq n$.

As in the Grassmann case, the differential $\psi_{\mathcal E}$ is
determined locally by the components
$$
A'_{ij}=\pi_i\circ\frac{\partial}{\partial z}\circ\pi_j\ , \quad
A''_{ij}=\pi_i\circ\frac{\partial}{\partial\bar z}\circ\pi_j\;,
$$
where $\pi_i\colon\, M\times H\to E_i$ is the orthogonal projection.
Note that by construction $A''_{ij}=-(A'_{ji})^*$.

Each of the subbundles $E_i$ of the trivial bundle $M\times H$ is
provided with the KM-structure which coincide with the complex
structure, induced from $M\times H$. Also the components $A'_{ij}$,
$A''_{ij}$ satisfy the harmonicity and holomorphicity conditions
similar to those in Sec.~\ref{ssec32}.

We introduce now an almost complex structure on the Hilbert--Schmidt
flag bundle $F_{\mathbf r}(H)$, analogous to the almost complex
structure $\mathcal J^2_\sigma$ from Sec.~\ref{ssec22}. As in the
Grassmannian case, an almost complex structure on $F_{\mathbf r}(H)$
is fixed by choosing the $(1,0)$-component in each of the summands
of
\begin{equation}
\label{flag_decomp} T^{\mathbb{C}}(F_{\mathbf
r}(H))\cong\bigoplus_{1\leq i<j\leq n}\left[\overline{E}_{i}E_{j}
\oplus\overline{E}_{j}E_{i}\right].
\end{equation}
(We recall that \eqref{flag_decomp} is just a concise form of
\eqref{decomp}.)

To define the almost complex structure $\mathcal J^2_\sigma$, we fix
an ordered subset $\sigma\subset\{1,\ldots,n\}$. Then for the
associated almost complex structure $\mathcal J^2_\sigma$ we choose
for $i,j\in\{1,\ldots,n\}$, $i<j$, the $(1,0)$-component in
$(i,j)$-summand in \eqref{flag_decomp}, equal to $\overline{E}_jE_i$
if $i,j\in\sigma$ or $i,j\notin\sigma$, and to $\overline{E}_iE_j$
if $i\in\sigma$, $j\notin\sigma$ or $i\notin\sigma$, $j\in\sigma$.

%-------------------------------------------------------------------------

\subsection{Twistor bundle over the Hilbert--Schmidt Grassmannian}
\label{ssec34}

%-------------------------------------------------------------------------

We construct now the Hilbert--Schmidt flag bundle over the
Grassmannian. Suppose that $\sigma$ is a fixed ordered subset in
$\{1,\ldots,n\}$ and set $r=\sum_{i\in\sigma}r_i$. Then we define
the \textit{Hilbert--Schmidt flag bundle}
$$
\pi_\sigma:F_{\mathbf r}(H)\longrightarrow G_r(H)
$$
by mapping
$$
\mathcal E=(E_1,\ldots,E_n)\longmapsto E:=\bigoplus_{i\in\sigma}E_i.
$$

With this definition we can prove the following

\begin{theorem}
\label{direct} Let $\sigma$ be an ordered subset in $\{1,\ldots,n\}$
such that $k\in\sigma$, $l\notin\sigma$. Then the map $\pi_\sigma$
of the Hilbert--Schmidt flag manifold $F_{\mathbf r}(H)$, provided
with the almost complex structure $\mathcal J^2_\sigma$ to $G_r(H)$:
$$
\pi_\sigma:F_{\mathbf r}(H)\longrightarrow G_r(H)
$$
is a twistor bundle. It means that for any $\mathcal
J^2_\sigma$-holomorphic map $\psi: M\to F_{\mathbf r}(H)$ its
projection $\varphi=\pi_\sigma\circ\psi: M\to G_r(H)$ is harmonic.
\end{theorem}

\begin{proof} To prove the theorem, it is sufficient to show that
for any moving flag $\mathcal E=(E_1,\ldots,E_n)$, corresponding to
a $\mathcal J^2_\sigma$-holomorphic map $\psi_{\mathcal E}:
M\to\mathcal F$, the bundle $E:=\bigoplus_{i\in\sigma}\,E_i$ is
harmonic. The holomorphicity of the map $\psi_{\mathcal E}$ means
that
$$
A'_{ij}=0=A''_{ji}\ \quad\text{if}\quad
\begin{cases} i>j\
&\text{and}\quad i,j\in\sigma\ \text{or}\ i,j\notin\sigma, \\
i<j\ &\text{and}\quad i\in\sigma,j\notin\sigma\ \text{or}\
i\notin\sigma, j\in\sigma.
\end{cases}
$$
We have to prove that the bundle $E$, i.e. that
$$
A'_E\circ\partial''_E=\partial''_{E^\perp}\circ A'_E.
$$

Take $s<t$ with $s\in\sigma$, $t\notin\sigma$. Then
\begin{multline}
0=\pi_t\sum_i\left[\frac{\partial}{\partial\bar
z}\pi_i\frac{\partial}{\partial z} - \frac{\partial}{\partial
z}\pi_i\frac{\partial}{\partial\bar z}\right]\pi_s=
\sum_i(A''_{ti}A'_{is}-A'_{ti}A''_{is})=\\
\sum_{i\notin\sigma}A''_{ti}A'_{is}-\sum_{i\in\sigma}A'_{ti}A''_{is}=
\left(\sum_{i\notin\sigma}A''_{ti}\right)
\left(\sum_{i\notin\sigma}A'_{is}\right)-
\left(\sum_{i\in\sigma}A'_{ti}\right)
\left(\sum_{i\in\sigma}A''_{is}\right)\\
=\pi_t\left(\partial''_{E^\perp}\circ
A'_E-A'_E\circ\partial''_E\right)\pi_s.
\end{multline}

Analogous relations are satisfied for $s>t$, which implies that
$A'_E\circ\partial''_E=\partial''_{E^\perp}\circ A'_E$, i.e. $E$ is
harmonic.
\end{proof}

In the case when $M$ is the Riemann sphere $\mathbb P^1$, it is
possible to prove a conversion of the theorem~\ref{direct}, based on
an infinite-dimensional extension of Birkhoff--Grothendieck
classification theorem for holomorphic vector bundles over $\mathbb
P^1$, considered in the next Section~\ref{ssec35}.

%-------------------------------------------------------------------------

\subsection{Infinite-dimensional version of the Birkhoff--Grothedieck theorem}
\label{ssec35}

%-------------------------------------------------------------------------

Recall that the standard Birkhoff--Grothendieck theorem asserts that
any holomorphic vector bundle $E$ of rank $d$ over $\mathbb P^1$ is
equivalent to the direct sum of holomorphic line bundles $\mathcal
O(\kappa_1)\oplus\ldots\oplus\mathcal O(\kappa_d)$ with some
integers $\kappa_1\geq\ldots\geq\kappa_d$.

In terms of transition functions a holomorphic vector bundle $E$
over $\mathbb P^1$ is determined by an invertible holomorphic
$d\times d$-matrix function $f$, defined in a neighbourhood $U_0\cap
U_\infty$ of the equator of the Riemann sphere $\mathbb P^1$. Here
the sets $U_0:=\mathbb P^1\setminus\{\infty\}$ and
$U_\infty:=\mathbb P^1\setminus\{0\}$ form the standard open
covering of $\mathbb P^1$. In these terms the Birkhoff--Grothendieck
theorem is equivalent to the factorization of the transition
function $f$ in the form
$$
f(z)=f_0(z)d(z)f_\infty(z)
$$
where $f_0$ (resp. $f_\infty$) is a holomorphic invertible matrix
function in $U_0$ (resp. $U_\infty$) and $d(z)$ is the diagonal
matrix function of the form
$$
d(z)=\text{diag}(z^{\kappa_1},\ldots,z^{\kappa_d}).
$$

There exists still another formulation of Birkhoff--Grothendieck
theorem in the form of Harder--Narasimhan filtration. Suppose that
$E$ is a holomorphic vector bundle $E$ of rank $d$ over $\mathbb
P^1$, identified with a subbundle of the trivial rank $d$ bundle
over $\mathbb P^1$. Then there exists a filtration of $E$ by
holomorphic subbundles
$$
0=\mathcal B_0\subset\mathcal B_1\subset\ldots\subset\mathcal B_k=E,
$$
having the quotients of the form
$$
\left.\mathcal B_i\right/\mathcal B_{i-1}\,\cong\,
\underbrace{L^{\beta_i}\oplus\ldots\oplus L^{\beta_i}}_{\mbox{$b_i$
times}},
$$
where $L^{\beta_i}$ is the $\beta_i$-th power of the tautological
line bundle $L$ over $\mathbb P^1$ and $\beta_1>\dots>\beta_k$. The
subbundle $\mathcal B_i$ can be defined as the smallest holomorphic
subbundle of $E$, containing the images of all meromorphic sections
of $E$ with divisors of degree, greater or equal to $\beta_i$. Using
the Hermitian metric on $\mathbb C^d$, we can identify the quotient
$\mathcal B_i/\mathcal B_{i-1}$ with the orthogonal complement $B_i$
of $\mathcal B_{i-1}$ in $\mathcal B_i$.

In the infinite-dimensional setting the Birkhoff--Grothendieck
theorem, proved in \cite{Ser1},\cite{Ser2}, implies that any
holomorphic Hilbert-space bundle $E$ over $\mathbb P^1$ with the
structure group $\text{GL}_{\text{HS}}(H)$, consisting of invertible
operators of the form $I+T$ is equivalent to the direct sum of a
finite number of holomorphic line bundles and a trivial
Hilbert-space bundle.

In terms of transition functions a holomorphic Hilbert-space bundle
$E$ is determined by a holomorphic operator function $F(z)=I+T(z)$
with values in the group $\text{GL}_{\text{HS}}(H)$, defined in the
neighborhood $U_0\cap U_\infty$. In these terms the
Birkhoff--Grothendieck theorem is equivalent to the factorization of
the transition function $F$ in the form
$$
F(z)=F_0(z)D(z)F_\infty(z)
$$
where $F_0(z)=I+T_0(z)$ (resp. $F_\infty(z)=I+T_\infty(z)$) is a
holomorphic operator function in $U_0$ (resp. $U_\infty$) with
values in $\text{GL}_{\text{HS}}(H)$ and $D(z)$ is the diagonal
operator function of the form
$$
D(z)=\text{diag}(z^{\kappa_1},z^{\kappa_2},\ldots)
$$
where $\kappa_1\geq\kappa_2\geq\ldots$ are some integers such that
all of them, apart from a finite number, are equal to 0.

In terms of Harder--Narasimhan filtration this theorem means that
there exists a filtration of $E$ by holomorphic subbundles
\begin{equation}
\label{har-nar} 0=\mathcal B_0\subset\mathcal
B_1\subset\ldots\subset\mathcal B_s=E,
\end{equation}
having the quotients of the form
$$
\left.\mathcal B_i\right/\mathcal B_{i-1}\,\cong\,
\underbrace{L^{\beta_i}\oplus\ldots\oplus L^{\beta_i}}_{\mbox{$b_i$
times}}\cong\,b_iL^{\beta_i}
$$
for $i=1,\ldots,s$, $i\neq k$, and
$$
\left.\mathcal B_k\right/\mathcal B_{k-1}\cong E_k=W_{\text{in}}
$$
where $\beta_1>\dots>\beta_s$. Using the Hermitian metric on $H$, we
can identify the quotient $\mathcal B_i/\mathcal B_{i-1}$ with the
orthogonal complement $B_i$ of $\mathcal B_{i-1}$ in $\mathcal B_i$.
Note that the induced complex structure on $\mathcal B_i/\mathcal
B_{i-1}$ coincides with the complex KM-structure on $B_i$.

%-------------------------------------------------------------------------

\subsection{Twistor description of harmonic spheres in Hilbert--Schmidt
Grassmannians}
\label{ssec36}

%-------------------------------------------------------------------------

Now we are ready to formulate the converse of theorem \ref{direct}.

\begin{theorem}
\label{converse} Let $\varphi: \mathbb P^1\to G_r(H)$ be a harmonic
map. Then there exist a Hilbert--Schmidt flag bundle
$$
\pi_\sigma:F_{\mathbf r}(H)\longrightarrow G_r(H)
$$
and a $\mathcal J^2_\sigma$-holomorphic map $\psi: \mathbb P^1\to
F_{\mathbf r}(H)$ such that $\varphi$ coincides with the projection
$\pi_\sigma\circ\psi$ of the map $\psi$.
\end{theorem}

\begin{proof} The proof of this theorem goes on along the same lines
as in the finite-dimensional case (cf. \cite{Bur-Sal}).

Associate with the harmonic map $\varphi: \mathbb P^1\to G_r(H)$ the
bundle $E$ over $\mathbb P^1$ and provide it with the complex
KM-structure. With respect to this complex structure $E$ becomes a
holomorphic Hilbert-space bundle. So by the Birkhoff--Grothendieck
theorem (in Harder--Narasimhan form) there exists a filtration of
$E$ by holomorphic subbundles
$$
0=\mathcal B_0\subset\mathcal B_1\subset\ldots\subset\mathcal B_s=E,
$$
with the quotients of the form
$$
\left.\mathcal B_i\right/\mathcal B_{i-1}\,\cong\,b_iL^{\beta_i}
$$
for $i=1,\ldots,s$, $i\neq k$, and $\left.\mathcal
B_k\right/\mathcal B_{k-1}\cong E_k=W_{\text{in}}$, such that
$\beta_1>\dots>\beta_s$ (where we set $\beta_k=0$, $b_k=1$).

We can construct an analogous filtration for the orthogonal
complement $E^\perp$ of $E$ in $\mathbb P^1\times H\to\mathbb P^1$:
$$
0=\mathcal C_0\subset\mathcal C_1\subset\ldots\subset\mathcal C_t=
E^\perp
$$
with quotients of the form
$$
\left.\mathcal C_i\right/\mathcal C_{i-1}\,\cong\,c_iL^{\gamma_i}
$$
for $i=1,\ldots,t$, $i\neq l$, and
$$
\left.\mathcal C_l\right/\mathcal C_{l-1}\cong E_l=W_{\text{out}}
$$
such that $\gamma_1>\dots>\gamma_t$ (where we also set $\gamma_l=0$,
$c_l=1$). We identify again the quotient $\mathcal C_i/\mathcal
C_{i-1}$ with the orthogonal complement $C_i$ of $\mathcal C_{i-1}$
in $\mathcal C_i$.

We collect now the subbundles $B_1,\dots,B_s,C_1,\dots,C_t$ into a
single collection of $n=s+t$ subbundles, denoted by $E_1,\dots,E_n$,
so that each of $E_i$ is isomorphic to the direct sum of $e_i$
copies of $L^{\delta_i}$ and $\delta_1\leq\dots\leq\delta_n$. (If
for some $j$ we have $\delta_j=\delta_{j+1}$, we arrange the
associated subbundles $E_j$, $E_{j+1}$ in such a way that $E_j$
corresponds to some $B_p$ and $E_{j+1}$ to some $C_q$.) We introduce
the subset $\sigma\subset\{1,2,\dots,n\}$, uniquely defined by the
equalities
$$
E=\bigoplus_{i\in\sigma}E_i,\quad E^\perp=\bigoplus_{i\notin\sigma}
E_i.
$$

Denote by $\mathcal E:=(E_1,\dots,E_n)$ the moving flag, determined
by the subbundles $E_1,\ldots,E_n$, and by $\psi_{\mathcal E}:
\mathbb P^1\to F_{\mathbf r}(H)$ the corresponding map, associated
with $\mathcal E$. We have to prove that this map is $\mathcal
J^2_\sigma$-holomorphic. In other words, we should prove that
$$
A'_{ij}=0=A''_{ji}\ \quad\text{if}\quad
\begin{cases} i>j\
&\text{and}\quad i,j\in\sigma\ \text{or}\ i,j\notin\sigma, \\
i<j\ &\text{and}\quad i\in\sigma,j\notin\sigma\ \text{or}\
i\notin\sigma, j\in\sigma.
\end{cases}
$$

Suppose first that $i>j$ and $i,j\in\sigma$. Then
$\delta_i>\delta_j$ and the subbundle $E_i$ is contained in some
holomorphic subbundle $\mathcal B_p$ of $E$, orthogonal to $E_j$. It
follows that $A_{ji}''=0$, which implies also that $A'_{ij}=0$. The
case $i,j\notin\sigma$ is treated in a similar way.

Suppose next that $i<j$ and $i\notin\sigma, j\in\sigma$. Then
$E_j=B_p$ for some $B_p\subset\mathcal B_p$. Since $E$ is harmonic,
it follows that differential $dz\otimes A'_E$ is holomorphic (cf.
equations \eqref{harm-eq1}, \eqref{harm-eq2} in Sec.~\ref{ssec21}).
Here, $A'_E$ is considered as a section of the holomorphic bundle
$\text{Hom}(E,E^\perp)$. Since the image $A'_E(\mathcal B_p)$ is
spanned by meromorphic sections of $E^\perp$ with divisors of
degree, greater than $\delta_j+1$, we have
$$
A'_E(E_j)\subset\bigoplus_{q\notin\sigma,\,q>j} E_q.
$$
Hence, $A'_{ij}=0$ for $i<j$, implying also that $A''_{j\,i}=0$.The
case $i\in\sigma,j\notin\sigma$ is treated in a similar way, using
the fact that the subbundle $E^\perp$ is harmonic if $E$ has this
property.
\end{proof}

%-------------------------------------------------------------------------

\subsection{Reducing the length of harmonic bundles}
\label{ssec37}

%-------------------------------------------------------------------------

Consider again the Harder--Narasimhan filtration \eqref{har-nar}
from Sec.~\ref{ssec35}:
$$
0=\mathcal B_0\subset\mathcal B_1\subset\ldots\subset\mathcal B_s=E
$$
with the quotients
$$
\left.\mathcal B_i\right/\mathcal B_{i-1}\,\cong\,
\underbrace{L^{\beta_i}\oplus\ldots\oplus L^{\beta_i}}_{\mbox{$b_i$
times}}\cong\,b_iL^{\beta_i}
$$
for $i=1,\ldots,s$, $i\neq k$, and $\left.\mathcal
B_k\right/\mathcal B_{k-1}\cong E_k=W_{\text{in}} $ with
$\beta_1>\dots>\beta_s$. The number $\beta_1-\beta_s$ is called the
\textit{length} of the bundle $E$.

Burstall and Salamon in \cite{Bur-Sal} have proposed a procedure of
reducing the length of the original bundle $E$, which is similar to
Uhlenbeck's ''subtracting-a-uniton'' construction or Valli's
construction, reducing the energy of a harmonic bundle. The
Burstall--Salamon's procedure is immediately extended to the case of
infinite-dimensional harmonic bundles. Since the argument
essentially repeats the corresponding argument from \cite{Bur-Sal}
we formulate only the main result.

\begin{theorem}
\label{bur-sal3} For a given harmonic bundle $E \subset H \times
P^{1}$ there exists a finite sequence $E^{0}$,$E^{1},\ldots,E^{N} =
E$ of harmonic bundles such that
\begin{enumerate}
\item[(i)] $0=\textrm{length}(E^{0}) < \textrm{length}(E^{1})< \ldots
<\textrm{length}(E^{N}) = \beta_{1}-\beta_{s}$,
\item[(ii)]
$E^{i-1}$ is obtained from $E^{i}$ or $(E^{i})^{\perp}$ by removing a
holomorphic or anti-holomorphic subbundle.
\end{enumerate}
\end{theorem}

According to this theorem, everything is reduced to the case of
harmonic bundles of length zero which are described by the following

\begin{proposition}
\label{bur-sal4} Any harmonic bundle of length zero has the form
$$
E=F\ominus F_{1}
$$
where $F_{1}$ and $F$ are holomorphic subbundles of $\mathbb P^1
\times H$, satisfying the condition $\frac{\partial}{\partial z}
\Gamma(F_{1}) \subset \Gamma(F)$.
\end{proposition}

%\section*{Bibliography}

\end{document}